\documentclass{amsart}

\usepackage[utf8]{inputenc}
\usepackage{lmodern}
\usepackage[T1]{fontenc}

\usepackage{amsmath,amsthm,amsfonts,amssymb}
\usepackage{mathtools}	% Para los pre-índices de los operadores fraccionarios

\usepackage{enumerate}  % Manage lists
\usepackage{enumitem}	% label en enumerates
\usepackage[colorlinks=true,urlcolor=blue,
citecolor=red,linkcolor=blue,linktocpage,pdfpagelabels,
bookmarksnumbered,bookmarksopen]{hyperref}  % Hyperlinks
\usepackage{stmaryrd}  % For double brackets

\numberwithin{equation}{section}

\newtheorem{Theorem}{Theorem}[section]
\newtheorem{Lemma}[Theorem]{Lemma}
\newtheorem{Property}[Theorem]{Property}
\theoremstyle{definition}
\newtheorem{Definition}{Definition}[section]
\theoremstyle{remark}
\newtheorem{Remark}[Definition]{Remark}

\newcommand{\R}{\mathbb{R}}  % real numbers
\newcommand{\N}{\mathbb{N}}  % natural numbers
\newcommand{\C}{\mathbb{C}}  % complex numbers

\newcommand{\F}{\mathcal{F}}	% Transformada de Fourier
\newcommand{\bH}{{\bf H}}
\newcommand{\bL}{{\bf L}}
\newcommand{\bC}{{\bf C}}

\DeclareMathOperator{\sgn}{sgn}	% Función signo

\begin{document}

\title[Fractional hyperbolic problems in the theory of viscoelasticity]{Variational formulation for fractional hyperbolic problems in the theory of viscoelasticity}

\author[J. Bravo-Castillero]{Julián Bravo-Castillero}
\address{{\rm (J. Bravo-Castillero)} Instituto de Investigaciones en Matem\'aticas Aplicadas y en Sistemas \\ Universidad Nacional Aut\'onoma de M\'exico \\ Unidad Académica del IIMAS en el Estado de Yucatán \\ Parque Científico y Tecnológico de Yucatán, CP 97302, Mérida, Yucatán, México}
\email{julian@mym.iimas.unam.mx}

\author[L.~F. López Ríos]{Luis Fernando López Ríos} 
\address{{\rm (L.~F. López Ríos)} Instituto de Investigaciones en Matem\'aticas Aplicadas y en Sistemas \\ Universidad Nacional Aut\'onoma de 
M\'exico\\Circuito Escolar s/n, Ciudad Universitaria, C.P. 04510, Cd. de M\'{e}xico (Mexico)}
\email{luis.lopez@mym.iimas.unam.mx}

% \date{\today}

\begin{abstract}
  In this article a theoretical framework for problems involving fractional equations of hyperbolic type arising in the theory of viscoelasticity is presented. Based on the Galerkin method, a variational problem of the fractionary viscoelasticity is studied. An appropriate functional setting is introduced in order to establish the existence, uniqueness and a priori estimates for weak solutions. This framework is developed
in close concordance with important physical quantities of the theory of
viscoelasticity.
\end{abstract}



\maketitle

% -------------------------------------------
% -------------------------------------------

%\tableofcontents

%-------------------------------------------

\section{Introduction}

Viscoelastic materials combine properties of elastic solids and viscous fluids. Elastic materials return to their original configuration when the application of a force ends. However, the deformation of a viscous fluid increases over time when a force is applied. The mathematical modeling of viscoelastic materials is based on the continuum mechanics theory. For instance, for a linear viscoelastic material with a Kelvin-Voigt constitutive law, the elastic contribution of the stress is proportional to the strain; while the viscous part of the stress is proportional to the standard derivative of the strain with time. In \cite{banks_hu_kenz_2011}, a brief review of the viscoelasticity theory with detailed explanations and examples can be seen. In \cite{Yo_etal_2019}, in addition to analyzing various theoretical models, important practical applications are reviewed (structural systems, marine pipelines, aerospace industry, biomechanics and nanoresonators). These mathematical models are represented by systems of standard partial differential equations that have been extensively studied. For example, existence, uniqueness and stability of the solution of related problems appear in \cite{FaMo92}. 

However, in general, the mathematical models above are not accurate enough for more complex viscoelastic materials such as many polymers \cite{BaTo83-1} and biological tissues \cite{Ma97}. In \cite{BaTo83-1} fractional calculus were considered to construct stress-strain relationships for viscoelastic materials, while in \cite{Ma97} viscoelastic properties of human soft tissues were studied considering a Kelvin-Voigt model with a derivative of the fractional order of the deformation. In both of these models were found that fractional derivatives better approximate experiments than classical ones. This agreement with the experiments and the simplicity of the model by introducing very few empirical parameters, make the  application of fractional calculus to describe viscoelastic phenomena very attractive. In \cite{Ma2010} can be found many more examples of the use of fractional calculus in the theory of viscoelasticity.

% -------------------------------------------

\subsection{Main result}
\label{sec:main-results}

In 1983, Bagley and Torvik suggested to use fractional derivatives to construct stress-strain relationships for viscoelastic materials. In the early stage the use of fractional calculus in this context was based on phenomenological arguments \cite{BaTo83-1}, but then it was linked to the molecular theory for dilute polymer solutions developed by Rouse \cite{BaTo83-2}. This theoretical basis for the fractional constitutive relationships gave confidence in their use to describe accurately the mechanical properties of viscoelastic materials. One of these fractional stress-strain constitutive relationships is the Kelvin-Voigt type, associated to the following initial boundary problem of hyperbolic type:   
\begin{equation} \label{eq:main_1}
  \left\{
  \begin{aligned}
    &\rho(x) \frac{\partial^2 u}{\partial t^2} - \frac{\partial}{\partial x_i} \left[ B_{ij}(x) \frac{\partial}{\partial x_j} \prescript{C}{0}{D}_t^\alpha u \right] - \frac{\partial}{\partial x_i} \left[ A_{ij}(x) \frac{\partial u}{\partial x_j} \right] = f(x,t), \\
    &\hphantom{aaaaaaaaaaaaaaaaaaaaaaaaaaaaaaaaaaaaaaaa}(x,t) \in \Omega \times (0,T), \\
    &u(x,t) = 0, \quad (x,t) \in \partial \Omega \times (0,T), \\
    &u(x,0) = g(x), \quad \frac{\partial u}{\partial t}(x,0) = h(x), \quad x \in \Omega, 
  \end{aligned}
  \right.
\end{equation}
where the summation over the repeating indices is assumed. Here $\Omega \subset \R^n$ is the domain occupied by the viscoelastic material and $u$ the unknown displacement vector field. The material functions are: the mass density $\rho$, the elastic tensor $A_{jk}$ and the fractional viscosity tensor $B_{jk}$, with fraction $\alpha \in (0,1)$. The vector function $f$ describe an external force on the material, and the phenomenon is studied in the time interval $(0,T)$, with initial displacement an velocity fields $g$ and $h$, respectively. The operator $\prescript{C}{0}{D}_t^\alpha$ is the Caputo fractional derivative in time.

In this article, we develop a variational method to study the fractional Kelvin-Voigt model (\ref{eq:main_1}), providing a basis for a solid mathematical framework to study some important properties of the original model proposed by Bagley and Torvik. These properties includes the balance of the kinetic and elastic energies together with the energy dissipated by the fractional viscosity. This variational study is necessary to guarantee the convergence of the Galerkin method and to find error estimates in the Galerkin finite element approximation method to numerically solve (\ref{eq:main_1}). The main result is the following theorem, which establishes the existence, uniqueness and \emph{a priori} estimates for weak solutions to (\ref{eq:main_1}).
\begin{Theorem} \label{thm:main}
  Under hypothesis \ref{H1}-\ref{H3} in Section~\ref{sec:existence}, there exists a unique weak solution to (\ref{eq:main_1}) and a constant $C=C(\nu,\rho_0,\Omega,T,\alpha)$ such that
  \begin{multline} \label{eq:energy_est}
    \| u \|_{L^\infty(0,T;\bH_0^1(\Omega))} + \| u_t \|_{L^\infty(0,T;\bL^2(\Omega))} + \| u \|_{H_0^{\alpha/2}(0,T;\bH_0^1(\Omega))} \\
   \le \ C \left( \| f \|_{L^2(0,T;\bH^{-1}(\Omega))} + \| g \|_{\bH_0^1(\Omega)} + \| h \|_{\bL^2(\Omega)} \right).
  \end{multline}
\end{Theorem}

This article is organized as follows. In Section~2 the Riemann-Liouville and Caputo fractional derivatives are introduced, these operators are well documented and their main properties can be found in \cite{SaKiMa93}. In Section~3 we use the $L^2$ theory of the Fourier transform to construct appropriate time fractional spaces and their associate variational properties, which are fundamental in the rest of the article. In Section~4 the time fractional theory is used to construct space-time fractional spaces and derivatives, which provide the appropriate variational framework to solve (\ref{eq:main_1}). Finally, in Section~5 we implement a Galerkin method to prove Theorem~\ref{thm:main}. Additionally, in the Appendix several useful properties of fractional derivatives and spaces used throughout the article are collected.       

% -------------------------------------------
% -------------------------------------------

\section{Fractional derivatives}

In this section we introduce the Riemann-Liouville and Caputo fractional derivatives, two of the most used fractional derivatives to describe viscoelastic materials. We refer to  \cite{SaKiMa93} for more details about the properties of these operators.

Let us first recall the usual notation $C_c^\infty(J)$ to represent the set of infinitely differentiable \emph{complex valued} functions, compactly supported in the, possibly unbounded, interval $J$. Given $a \in [-\infty,\infty)$ and $b \in (-\infty,\infty]$, with $a<b$ and $\varphi \in C^\infty_c(\R)$, the \emph{Riemann-Liouville fractional integrals} of order $\alpha>0$ is given by
 \begin{gather}
   \prescript{}{a}{I}_t^\alpha \varphi(t) = \frac{1}{\Gamma(\alpha)} \int_a^t \frac{\varphi(s)}{(t-s)^{1-\alpha}} \, ds, \label{eq:frac_integral1}\\
      \prescript{}{t}{I}_b^\alpha \varphi(t) = \frac{1}{\Gamma(\alpha)} \int_t^b \frac{\varphi(s)}{(s-t)^{1-\alpha}} \, ds, \label{eq:frac_integral2}
    \end{gather}
    where $\Gamma$ is the Gamma function. These integrals are called the \emph{left and right-sided fractional integrals}, respectively, and generalize the well-known Cauchy formula when $\alpha$ is a natural number. To complement the definition, we write  $\prescript{}{a}{I}_t^0 = \prescript{}{t}{I}_b^0  := I$, where $I$ is the identity operator.

    The \emph{Riemann-Liouville fractional derivatives} of order $\alpha$, with $m-1 < \alpha \le  m$ for certain $m \in \N$, are defined by
    \begin{equation}
      \label{eq:frac_der}
      \prescript{}{a}{D}_t^\alpha \varphi(t) := \frac{d^m}{dt^m} \circ \prescript{}{a}{I}_t^{m-\alpha} \varphi(t), \quad \prescript{}{t}{D}_b^\alpha \varphi(t) := (-1)^m \frac{d^m}{dt^m} \circ \prescript{}{t}{I}_b^{m-\alpha} \varphi(t).
    \end{equation}   
When $\alpha = m$, an integer, these derivatives coincide, up to a sign, with the usual derivative of order $m$; otherwise for $\alpha$ a fraction,
\begin{equation*}
  \prescript{}{a}{D}_t^\alpha \varphi(t) = \frac{1}{\Gamma(m-\alpha)} \frac{d^m}{dt^m} \int_a^t \frac{\varphi(s)}{(t-s)^{\alpha+1-m}} \, ds,
\end{equation*}
with a similar expression for the right-sided fractional derivative. To complement the definitions, we assume $\prescript{}{a}{D}_t^0 = \prescript{}{t}{D}_b^0  := I$.

If we interchange the operators in (\ref{eq:frac_der}) we obtain the so-called \emph{Caputo fractional derivatives} of order $\alpha$:
\begin{equation} \label{eq:frac_der_Caputo}
  \prescript{C}{a}{D}_t^\alpha \varphi(t) := \prescript{}{a}{I}_t^{m-\alpha} \circ \frac{d^m \varphi}{dt^m}(t), \quad \prescript{C}{t}{D}_b^\alpha \varphi(t) := (-1)^m \prescript{}{t}{I}_b^{m-\alpha} \circ \frac{d^m \varphi}{dt^m}(t),
\end{equation}
which coincide, up to a sign, with the usual derivative of order $m$, when $\alpha = m$ an integer, and
\begin{equation*}
  \prescript{C}{a}{D}_t^\alpha \varphi(t) = \frac{1}{\Gamma(m-\alpha)} \int_a^t \frac{\varphi^{(m)}(s)}{(t-s)^{\alpha+1-m}} \, ds,
\end{equation*}
when $\alpha$ is a fraction; there is a similar expression for the right-sided fractional derivative. As previously, we define $\prescript{C}{a}{D}_t^0 = \prescript{C}{t}{D}_b^0  := I$.

\begin{Remark}\label{rm:frac_AC}
  Observe that the fractional operators previously considered are applied to smooth functions, and thereby all the integrals and derivatives in consideration are well defined. However the operators can be further taken to more general functions; for instance if $[a,b]$ is a \emph{bounded} interval and $u \in L^1(a,b)$ then its fractional integrals (\ref{eq:frac_integral1}) and (\ref{eq:frac_integral2}) are well defined. Something similar can be done for the fractional derivatives: if $\alpha \in [0,1)$ and  $u \in AC[a,b]$, i.e. \emph{absolutely continuous}, then its fractional derivatives (\ref{eq:frac_der}) are also well defined and exist almost everywhere, with
\begin{equation} \label{eq:frac_der_AC}
  \prescript{}{a}{D}_t^\alpha u (t) = \frac{1}{\Gamma(1-\alpha)} \left[ \frac{u(a)}{(t-a)^\alpha} + \int_a^t \frac{u'(s)}{(t-s)^\alpha} \, ds \right].
\end{equation}
This identity provides a useful relation between the Riemann-Liouville and Caputo fractional derivatives:
\begin{equation} \label{eq:Caputo_frac_der}
  \prescript{C}{a}{D}_t^\alpha u (t) = \prescript{}{a}{D}_t^\alpha \left[ u(t) - u(a) \right],
\end{equation}
with similar identities for right-sided fractional derivatives and general $\alpha \ge 0$. For more details, see Section~2.3 in \cite{SaKiMa93}. Observe also that (\ref{eq:frac_der_AC}) and (\ref{eq:Caputo_frac_der}) imply that
\begin{equation}\label{eq:frac_der_cons}
    \prescript{}{a}{D}_t^\alpha 1 = \frac{1}{\Gamma(1-\alpha)} \frac{1}{(t-a)^\alpha}, \quad \prescript{C}{a}{D}_t^\alpha 1 = 0.
  \end{equation}
\end{Remark}

\begin{Remark}
  Concerning (\ref{eq:frac_der_cons}), it is important to mention that Theorem~\ref{thm:main} is still true if we consider the Riemann-Liouville fractional derivative (instead of the Caputo) in (\ref{eq:main_1}), but some care has to be taken. If $g \equiv 0$ the results remain unchanged, as both operators coincide in this case. On the other hand, if $g \not\equiv 0$, the fractional differentiation order $\alpha$ has to be restricted to the range $(0,1/2)$, as the Riemann-Liouville fractional derivative of a time-independent function is not cero; indeed, as we implement a $L^2$ variational theory to solve (\ref{eq:main_1}), in order to keep all the expressions in this space necessarily the range of $\alpha$ has to be restricted. Something different happens with the Caputo fractional derivative, which is zero in time-independent functions, giving it an advantage over the Riemann-Liouville fractional operator for modeling physical problems.
\end{Remark}

In order to develop some variational framework to study (\ref{eq:main_1}), we need to extend the fractional derivatives above to more general functions. To this end, in the next section we will use the $L^2$ theory of the Fourier transform together with an important property of Riemann-Liouville fractional derivatives: Property~\ref{Prop:App-Fourier_deri}.

% -------------------------------------------
% -------------------------------------------

\section{Fractional spaces}

Motivated by Property~\ref{Prop:App-Fourier_deri}, we use the $L^2$ theory of the Fourier transform to extend the fractional Riemann-Liouville derivatives from $C_c^\infty(\R)$ to a special Hilbert space, suitable to implement a variational framework to study (\ref{eq:main_1}).

The following definition of fractional spaces by using the Fourier transform is well-known, see for instance Chapter 15 in \cite{Ta07}.
\begin{Definition}
  For every $\alpha \ge 0$ we define the fractional Sobolev space
  \begin{equation*}
   H^\alpha(\R) = \{ u \in L^2(\R): |\omega|^\alpha \hat u \in L^2(\R)  \}.
  \end{equation*}
\end{Definition}

The fractional space $H^\alpha(\R)$ is a complex Hilbert space with inner product
\begin{equation} \label{eq:inner_product}
  (u,v)_\alpha := (u,v) + (|\omega|^\alpha \hat{u}, |\omega|^\alpha \hat{v}),
\end{equation}
where $(u,v) := \int_{-\infty}^\infty u \overline{v}$ is the inner product in $L^2(\R)$, with the bar denoting the complex conjugate. This inner product generate the norm
\begin{equation} \label{eq:norm}
  \| u \|_\alpha := (\| u \|_{L^2(\R)}^2 + |u|_\alpha^2)^{1/2},
\end{equation}where $| \cdot |_\alpha$ is the norm
\begin{equation*}
  |u|_\alpha := \| |\omega|^\alpha \hat{u} \|_{L^2(\R)}.
\end{equation*}
It is important to remark that these fractional spaces generalize the classical Sobolev spaces $H^m(\R)$, for nonnegative integers $m$, in particular $H^0(\R) = L^2(\R)$.

The Riemann-Liouville fractional derivatives of order $\alpha$ can be extended from $C_c^\infty(\R)$, in (\ref{eq:frac_der}), to every function  $u \in H^\alpha(\R)$ in the following way:
\begin{equation}\label{eq:frac_der_fourier}
  \prescript{}{-\infty}{D}_t^\alpha u (t) := \F^{-1}((i \omega)^\alpha \hat u(\omega)) \quad \text{and} \quad    \prescript{}{t}{D}_\infty^\alpha u (t) := \F^{-1}((-i \omega)^\alpha \hat u(\omega)).
\end{equation}
As a consequence of the Plancherel theorem, we can write down the Sobolev norm of $H^\alpha(\R)$ in the more classical way
\begin{equation}\label{eq:norm_2}
\| u \|_\alpha =  (\| u \|_{L^2(\R)}^2 + \| \prescript{}{-\infty}{D}_t^\alpha u \|_{L^2(\R)}^2)^{1/2}.
\end{equation}
We note that $\| \cdot \|_\alpha$ can also be defined with the right-sided fractional derivative $\prescript{}{t}{D}_\infty^\alpha$, nonetheless both norms are equal.

It is desirable to translate properties from $C_c^\infty(\R)$ to $H^\alpha(\R)$, and the next result justifies this approach. The proof of this result can be found in Chapter 15 of \cite{Ta07}.

\begin{Property} \label{Prop:density}
  The space $C_c^\infty(\R)$ is dense in $H^\alpha(\R)$.
\end{Property}

The fractional space $H^\alpha(\R)$ is formed by functions defined in the whole real line. For functions defined in a fixed open interval, it is still desirable to keep important properties of smooth functions in that interval. Thereby, and motivated by the previous proposition, we consider the following definition, see \cite{ErRo06,Gr85}.

\begin{Definition} \label{Def:fractional_interval}
  Let $\alpha \ge 0$ and $(a,b)$ be an open interval of $\R$, where $a \in [-\infty,\infty)$ and $b \in (-\infty,\infty]$ with $a<b$, so the interval could be unbounded. We define the space $H_0^\alpha(a,b)$ as the closure of $C_c^\infty(a,b)$ under the norm $\| \cdot \|_\alpha$.
\end{Definition}

The fractional space $H_0^\alpha(a,b)$ is a Hilbert subspace of $L^2(a,b)$ with the inner product and norm inherit under the limit of elements in $C_c^\infty(a,b)$: if $u,v \in H_0^\alpha(a,b)$ and $\{ \varphi_j \}_{j=1}^\infty, \{ \psi_j \}_{j=1}^\infty \subset C_c^\infty(a,b)$ are Cauchy sequences with the norm $\| \cdot \|_\alpha$ such that $\varphi_j \to u$ and $\psi_j \to v$ in $L^2(a,b)$, then
\begin{equation*}
  (u,v)_\alpha := \lim_{j \to \infty} (\varphi_j,\psi_j)_\alpha, \quad   \| u \|_\alpha := \lim_{j \to \infty} \| \varphi_j \|_\alpha,
\end{equation*}
where $(\varphi_j,\psi_j)_\alpha$ and $\| \varphi_j \|_\alpha$ are defined in (\ref{eq:inner_product}) and (\ref{eq:norm}), respectively. This inner product and norm are well-defined, in the sense that they don't depend on the approximating sequence and generalize the corresponding notions in $C^\infty_c(a,b)$, as can be deduced from the dominated convergence theorem. Observe also that $H_0^0(a,b) = L^2(a,b)$ and, by Property~\ref{Prop:density}, $H^\alpha_0(\R) = H^\alpha(\R)$ for every $\alpha \ge 0$.

In this article we will focus mainly in fractional spaces of order between 0 and 1/2, this choice will be clear later. For this variation range, Definition~\ref{Def:fractional_interval} take a more clear alternative form, allowing to identify better the functions in the fractional space, as we show in the next theorem. In what follows, we denote by $\widetilde u$ the \emph{continuation of $u$ by zero} outside of $(a,b)$.   

\begin{Theorem}\label{thm:H-alpha}
  For every $\alpha \in (0,1)$ the following properties are equivalent:
  \begin{enumerate}[label=(\roman*)]
  \item $u \in H_0^{\alpha/2}(a,b)$;
  \item $\widetilde u \in H^{\alpha/2}(\R)$;
  \item there exists $U \in H^{\alpha/2}(\R)$ such that restricted to $(a,b)$ is equal to $u$.
  \end{enumerate}
\end{Theorem}

\begin{proof}
  It is a simple matter to check that $(i) \Rightarrow (ii) \Rightarrow (iii)$. Let us prove $(ii) \Rightarrow (i)$, the proof of $(iii) \Rightarrow (ii)$ is similar.

  Suppose that $\widetilde{u} \in H^{\alpha/2}(\R)$. Then by Property~\ref{Prop:density}, there exists a sequence $\{ \varphi_j \}_{j=1}^\infty \subset C_c^\infty(\R)$ such that $\varphi_j \to \widetilde u$ in $H^{\alpha/2}(\R)$. Let us now use two standard approximation methods: cutting-off and regularization.
  
\noindent  \emph{Cutting-off}: consider $\widetilde \varphi_j := \varphi_j \chi_{(a+1/j,b-1/j)}$, where $\chi$ is the characteristic function and $j$ is big enough. We claim that $\widetilde \varphi_j \in H^{\alpha/2}(\R)$. Indeed, by using integration by parts we deduce that
  \begin{equation} \label{eq:fourier_comp}
    \begin{aligned}
      \widehat{\widetilde{\varphi}_j}(\omega) &= \frac{i}{\omega} \left( \widetilde{\varphi}_j(b)e^{-b\omega i} - \widetilde{\varphi}_j(a)e^{-a\omega i} - \int_a^b \varphi_j'(t)e^{-i\omega t}\, dt   \right) \\
      & = O\left( \frac{1}{|\omega|}  \right), \quad \text{as } |\omega| \to \infty.
    \end{aligned}
  \end{equation}
This estimate and the fact that $\alpha \in (0,1)$ imply $|\omega|^{\alpha/2} \widehat{\widetilde{\varphi}_j} \in L^2(\R)$, as we wished. Moreover, as $\varphi_j \to \widetilde u$ in $H^{\alpha/2}(\R)$, we can use the dominated convergence theorem to find that
  \begin{equation}\label{eq:cut-off}
    \widetilde \varphi_j \to \widetilde u \quad \text{in } H^{\alpha/2}(\R).
  \end{equation}

\noindent  \emph{Regularization}: let us now consider $\widetilde \varphi_j^* := \widetilde \varphi_j * \eta_{1/3j}$, where $\eta_j(t) = j\eta(jt)$ is a regularization sequence:
  \begin{equation*}
    \eta \in C^\infty_c (-1,1) \quad \text{with} \quad \int_{-\infty}^\infty \eta =1.
  \end{equation*}
 By well-known properties of the convolution, for every $j$, big enough, $\{ \widetilde \varphi_j^* \}_j \subset C_c^\infty(a,b)$; moreover, by (\ref{eq:cut-off}) it is a Cauchy sequence with the norm $\| \cdot \|_{\alpha/2}$ and 
  \begin{equation*}
    \widetilde \varphi_j^* \to u \quad \text{in } L^2(a,b).
  \end{equation*}
Thereby, we conclude that $u \in H_0^{\alpha/2}(a,b)$, and the proof is complete.
\end{proof}

As a consequence of previous theorem we deduce the following lemma, which states that the fractional space $H_0^{\alpha/2}(a,b)$ contains the classical Sobolev space of integer order:
\begin{equation*}
  H^1(a,b) := \{ u \in L^2(a,b): u' \in L^2(a,b) \}.
\end{equation*}

\begin{Lemma} \label{lm:regularity}
  For every $\alpha \in (0,1)$ we have $H^1(a,b) \subset H_0^{\alpha/2}(a,b)$.
\end{Lemma}
\begin{proof}
  Suppose $u \in H^1(a,b)$. By following the previous theorem, let us prove that $\widetilde u \in H^{\alpha/2}(\R)$. Indeed, note that estimation (\ref{eq:fourier_comp}), of the proof of the theorem, is still true for $\widetilde u$; and therefore $|\omega|^{\alpha/2}\widehat{\widetilde u} \in L^2(\R)$, as we wished. 
\end{proof}

\begin{Remark}\label{rm:H-alpha}
  \begin{enumerate}[label=(\roman*)]
  \item Theorem~\ref{thm:H-alpha} guarantees the equivalence of the three main methods to define fractional spaces in an \emph{interval}, see Chapter~1 in \cite{Gr85}.
  \item Lemma~\ref{lm:regularity} implies, in particular, that in any \emph{bounded interval} every $u \in C^1[a,b]$ belongs to the space $u \in H_0^{\alpha/2}(a,b)$. Thereby the functions in $H_0^{\alpha/2}(a,b)$, $\alpha \in (0,1)$, \emph{leave no trace} on the interval endpoints, so they are not necessarily equal to zero there; see Chapter 16 in \cite{Ta07} for more details.
  \end{enumerate}
\end{Remark}

\begin{Definition}
If $\alpha \in (0,1)$ and $u \in H_0^{\alpha/2}(a,b)$, we define the fractional Riemann-Liouville derivatives as
\begin{equation} \label{eq:frac_der_limit}
  \prescript{}{a}{D}_t^{\alpha/2} u := \lim_{j \to \infty} \prescript{}{-\infty}{D}_t^{\alpha/2} \varphi_j, \quad \prescript{}{t}{D}_b^{\alpha/2} u := \lim_{j \to \infty} \prescript{}{t}{D}_\infty^{\alpha/2} \varphi_j,
\end{equation}
where $\{\varphi_j\}_{j=1}^\infty \subset C_c^\infty(a,b)$ is a Cauchy sequence with the norm $\| \cdot \|_{\alpha/2}$ such that $\varphi_j \to u$ in $L^2(a,b)$  
\end{Definition}
The limits in (\ref{eq:frac_der_limit}) are taken in the $L^2(\R)$ sense and thereby both fractional derivatives belong to this space. By the dominated convergence theorem these fractional operators are well-defined in the sense that they don't depend on the approximating sequence.

In the next lemma we show that the fractional operator given by limit processes in (\ref{eq:frac_der_limit}) generalizes the classical definition of the fractional derivatives in (\ref{eq:frac_der}), for functions in $C^1[a,b]$, in a \emph{bounded} interval $(a,b)$; recall that by Remark~\ref{rm:frac_AC}, this last operator is well defined in this space. 

\begin{Lemma}\label{lm:frac_der}
  Let $\alpha \in (0,1)$ and $(a,b)$ a bounded interval. If $u \in C^1[a,b]$ then $u \in H_0^{\alpha/2}(a,b)$ and both fractional Riemann-Liouville derivatives in (\ref{eq:frac_der}) and (\ref{eq:frac_der_limit}) agree.
\end{Lemma}
\begin{proof}
  The key observation is that Property~\ref{Prop:App-Fourier_deri} $(i)$ is still true for \emph{piecewise} differentiable functions with \emph{compact support}, we leave the details to the reader and refer to Section~7.1 in \cite{SaKiMa93}. Thereby if we apply this property to the piecewise differentiable function with compact support $\widetilde{u}$, we have
\begin{equation*}
  \prescript{}{-\infty}{D}_t^{\alpha/2} \widetilde{u}(\omega) = \F^{-1}((i \omega)^{\alpha/2} \widehat{\widetilde{u}}(\omega)),
\end{equation*}
where the fractional derivative here is the classical one, in (\ref{eq:frac_der}). On the other hand,  Theorem~\ref{thm:H-alpha} implies that $\widetilde{u} \in H^{\alpha/2}(\R)$ (see also Remark~\ref{rm:H-alpha} $(ii)$) and by (\ref{eq:frac_der_fourier}) we have
\begin{equation*}
  \prescript{}{-\infty}{D}_t^{\alpha/2} \widetilde{u}(\omega) = \F^{-1}((i \omega)^{\alpha/2} \widehat{\widetilde{u}}(\omega)),
\end{equation*}
where the fractional derivative here is the one in (\ref{eq:frac_der_fourier}). We conclude that these fractional derivatives are equal. Moreover, by similar arguments to those in the proof of Theorem~\ref{thm:H-alpha}, the fractional derivative in the last equation coincide with the one by the limit in (\ref{eq:frac_der_limit}), and the proof is complete.
\end{proof}
From the proof of Theorem~\ref{thm:H-alpha}, we observe that the sequence $\{ \widetilde \varphi_j^* \}_{j=1}^\infty \subset C_c^\infty(a,b)$ is such that 
    \begin{equation*}
      \prescript{}{a}{D}_t^{\alpha/2} \widetilde \varphi_j^* \to \prescript{}{a}{D}_t^{\alpha/2} u, \quad  \prescript{}{t}{D}_b^{\alpha/2} \widetilde \varphi_j^* \to \prescript{}{t}{D}_b^{\alpha/2} u, \quad \text{in } L^2(\R).
    \end{equation*}
    These approximations are needed in order to prove energy estimates in Theorem~\ref{thm:energy_equiv}.

We provide next some fractional variational formulae that will be a fundamental ingredient to develop the Galerkin method in Section~\ref{sec:existence}. 

\begin{Theorem}
  Let $\alpha \in (0,1)$ and $(a,b)$ a bounded interval.
  \begin{enumerate}[label=(\roman*)]
  \item If $u,v \in H_0^{\alpha/2}(a,b)$ then
    \begin{equation} \label{eq:frac_var_1}
      \int_a^b \prescript{}{a}{D}_t^{\alpha/2} u v \, dt = \int_a^b u \prescript{}{t}{D}_b^{\alpha/2}v \, dt;
    \end{equation}
  \item If $u \in C^2[a,b]$ and $v \in H_0^{\alpha/2}(a,b)$ then
    \begin{equation}\label{eq:frac_var_2}
      \int_a^b \prescript{C}{a}{D}_t^\alpha u v \, dt = \int_a^b \prescript{C}{a}{D}_t^{\alpha/2} u \prescript{}{t}{D}_b^{\alpha/2} v \, dt.
    \end{equation}
  \item If $u,v \in L^2(a,b)$ then
    \begin{gather}
      \int_a^b \prescript{}{a}{I}_t^{\alpha/2} u v \, dt = \int_a^b u \prescript{}{t}{I}_b^{\alpha/2}v \, dt, \\
      \int_a^b \prescript{}{a}{I}_t^\alpha u v \, dt = \int_a^b \prescript{}{a}{I}_t^{\alpha/2} u \prescript{}{t}{I}_b^{\alpha/2} v \, dt. \label{eq:frac_var_I}
    \end{gather}
  \end{enumerate}
\end{Theorem}

\begin{proof}
  If $u,v \in C_c^\infty(a,b)$, we provide a proof of $(i)$ in the Appendix, see Property~\ref{Prop:App-parts}. In general, for $u,v \in H_0^{\alpha/2}(a,b)$ we can deduce $(i)$ by a limit argument, as by definition all the elements if this set can be approximate by sequences of smooth functions (recall also Lemma~\ref{lm:frac_der}).

  Let us now prove $(ii)$. By (\ref{eq:Caputo_frac_der}) and the semigroup property (\ref{eq:App-simi-der}),
  \begin{equation} \label{eq:frac_var_3}
\prescript{C}{a}{D}_t^\alpha u (t) = \prescript{}{a}{D}_t^\alpha \left[ u(t) - u(a) \right] = \prescript{}{a}{D}_t^{\alpha/2} \prescript{}{a}{D}_t^{\alpha/2} \left[ u(t) - u(a) \right].
\end{equation}
We claim that $\prescript{}{a}{D}_t^{\alpha/2} \left[ u(t) - u(a) \right] \in H^1(a,b)$. Indeed, by (\ref{eq:frac_der_AC}) and (\ref{eq:frac_der}),
\begin{align*}
  \frac{d}{dt} \prescript{}{a}{D}_t^{\alpha/2} \left[ u(t) - u(a) \right] &= \prescript{}{a}{D}_t^{\alpha/2} u'(t) \\
                                                                          &= \frac{1}{\Gamma(1-\alpha/2)} \left[ \frac{u'(a)}{(t-a)^{\alpha/2}} + \int_a^t \frac{u''(s)}{(t-s)^{\alpha/2}} \, ds \right],
\end{align*}
and this last term belongs to $L^2(a,b)$. Therefore by Lemma~\ref{lm:regularity}, we can use (\ref{eq:frac_var_1}) together with (\ref{eq:frac_var_3}) to deduce (\ref{eq:frac_var_2}).

The proof of $(iii)$ is similar to the previous two literals, but using instead Property~\ref{Prop:App-bound} together with (\ref{eq:App-frac_int_parts_2}) and (\ref{eq:App-semi}).
\end{proof}

\begin{Theorem} \label{thm:energy_equiv}
  Let $\alpha \in (0,1)$, $(a,b)$ a bounded interval and $u$ a \emph{real valued} function. Then
  \begin{enumerate}[label=(\roman*)]
  \item if $u \in H_0^{\alpha/2}(a,b)$,
    \begin{equation}\label{eq:energy_equiv_D}
      \| \prescript{}{a}{D}_t^{\alpha/2} u \|_{L^2(\R)}^2 = \frac{1}{\cos(\alpha \pi/2)}  \int_a^b \prescript{}{a}{D}_t^{\alpha/2} {u} \prescript{}{t}{D}_b^{\alpha/2} u \, dt;
    \end{equation}
  \item if $u \in L^2(a,b)$,
    \begin{equation}\label{eq:energy_equiv_I}
    \| \prescript{}{a}{I}_t^{\alpha/2} u \|_{L^2(a,b)}^2 = \frac{1}{\cos(\alpha \pi/2)}  \int_a^b \prescript{}{a}{I}_t^{\alpha/2} {u} \prescript{}{t}{I}_b^{\alpha/2} u \, dt,
  \end{equation}
  where the Riemann-Liouville fractional integrals are taken in the sense of classical formulae (\ref{eq:frac_integral1}) and (\ref{eq:frac_integral2}), see Remark~\ref{rm:frac_AC}.
  \end{enumerate}
\end{Theorem}

\begin{proof}
  Let us prove $(i)$, the proof of $(ii)$ is similar. First note that its is enough to prove (\ref{eq:energy_equiv_D}) for $u \in C_c^\infty(a,b)$, as for general $u$ we just recall Lemma~\ref{lm:frac_der} and the comments after its proof. So suppose $u \in C_c^\infty(a,b)$. To prove (\ref{eq:energy_equiv_D}) we follow \cite{ErRo06} and use the Fourier transform, we provide here the details for completeness. Let us start by recalling the following well-known property of the Fourier transform
  \begin{equation*}
    \int_{-\infty}^\infty \hat u v = \int_{-\infty}^\infty u \hat v,
  \end{equation*}
with a similar property for the inverse Fourier transform. As $u \in C_c^\infty(a,b)$, this identity and Property~\ref{Prop:App-Fourier_deri} imply
  \begin{align*}
    \int_a^b \prescript{}{a}{D}_t^{\alpha/2} {u} \prescript{}{t}{D}_b^{\alpha/2} u \, dt &= \int_{-\infty}^\infty \prescript{}{-\infty}{D}_t^{\alpha/2} {u} \prescript{}{t}{D}_\infty^{\alpha/2} u \, dt \\
    &= \int_{-\infty}^\infty (i\omega)^{\alpha/2} \hat u \overline{(-i\omega)^{\alpha/2} \hat u} \, d\omega.
  \end{align*}
  On the other hand, by (\ref{eq:complex_power})
  \begin{equation*}
    (-i\omega)^{\alpha/2} = (i\omega)^{\alpha/2} e^{-\frac{\alpha \pi i}{2} \sgn \omega},
  \end{equation*}
  and therefore
  \begin{align*}
    &\int_a^b \prescript{}{a}{D}_t^{\alpha/2} {u} \prescript{}{t}{D}_b^{\alpha/2} u \, dt  \\
    &= \int_{-\infty}^\infty (i\omega)^{\alpha/2} \hat u \overline{(i\omega)^{\alpha/2} \hat u} e^{\frac{\alpha \pi i}{2} \sgn \omega} \, d\omega \\
    &= \int_{-\infty}^0 (i\omega)^{\alpha/2} \hat u \overline{(i\omega)^{\alpha/2} \hat u} \, d\omega e^{-\frac{\alpha \pi i}{2}} + \int_0^\infty (i\omega)^{\alpha/2} \hat u \overline{(i\omega)^{\alpha/2} \hat u} \, d\omega e^{\frac{\alpha \pi i}{2}}  \\
    &= \cos(\alpha \pi/2) \int_{-\infty}^\infty (i\omega)^{\alpha/2} \hat u \overline{(i\omega)^{\alpha/2} \hat u} \, d\omega \\
    &\phantom{==} + i \sin(\alpha \pi/2) \left[ \int_0^\infty (i\omega)^{\alpha/2} \hat u \overline{(i\omega)^{\alpha/2} \hat u} \, d\omega - \int_{-\infty}^0 (i\omega)^{\alpha/2} \hat u \overline{(i\omega)^{\alpha/2} \hat u} \, d\omega  \right] \\
    &= \cos(\alpha \pi/2) \int_{-\infty}^\infty |\omega|^\alpha |\hat u|^2 \, d\omega,
  \end{align*}
  where the last equality is a consequence of the fact that $u$ is real valued, and thereby $\overline{\hat u(-\omega)} = \hat u(\omega)$. We conclude by Plancherel theorem.
\end{proof}

To finish this section, we show a fractional version of the classical Poincaré inequality.

\begin{Theorem}[Fractional Poincaré inequality] \label{thm:frac_Poincare}
  There exists $C>0$ such that, for every $u \in H_0^{\alpha/2}(a,b)$, we have
  \begin{equation*}
    \| u  \|_{L^2(a,b)} \le C \| \prescript{}{a}{D}_t^{\alpha/2} u \|_{L^2(a,b)}.
  \end{equation*}
\end{Theorem}

\begin{proof}
  By Property~\ref{Prop:App-inverse} and \ref{Prop:App-bound} we deduce that there exists a positive constant $C$ such that
  \begin{equation*}
    \| u \|_{L^2(a,b)} = \| \prescript{}{a}{I}_t^{\alpha/2} \prescript{}{a}{D}_t^{\alpha/2} u \|_{L^2(a,b)} \le C \| \prescript{}{a}{D}_t^{\alpha/2} u \|_{L^2(a,b)} 
  \end{equation*}
  for all $u \in H_0^{\alpha/2}(a,b)$, as we stated.
\end{proof}

% -------------------------------------------
% -------------------------------------------

\section{Space-time fractional spaces}

In this section we present suitable space-time fractional spaces to implement a Galerkin method to solve (\ref{eq:main_1}) in the next section. This approach is based on classical ideas coming from Fourier analysis, specially the separation of variables method.

We start by recalling some notions of strong measurability and integrability of functions with values in Hilbert spaces. In what follows $H$ represents a separable complex Hilbert space with inner product $(\cdot,\cdot)$, norm $\| \cdot \|$ and orthogonal basis $\{ w_k \}_{k=1}^\infty$. To fix ideas, we are mainly interested in the Sobolev space $\bH_0^1(\Omega)$ of \emph{vector fields} in $\Omega$, a bounded domain of $\R^n$ with Lipschitz boundary (see next section); a possible orthogonal basis for this space can be assembled with the eigenfuntions of the Laplace operator in the domain.

We recall also some well-known facts about measurable functions with values in Hilbert spaces, for more details see Sections~V.4 and V.5 of \cite{Yo80} and Appendix E.5 of \cite{Ev2010}. Let $\alpha \ge 0$, $(a,b)$ be an open interval of $\R$, where $a \in [-\infty,\infty)$ and $b \in (-\infty,\infty]$ with $a<b$, and $H$ be a separable complex Hilbert space.
  \begin{enumerate}[label=(\roman*)]
  \item A function $\varsigma:(a,b) \to H$ is called \emph{simple} if it has the form
    \begin{equation}\label{eq:simple_funct_1}
      \varsigma(t) = \sum_{j=1}^m e_j \chi_{E_j}(t) w_j, \quad t \in (a,b),
    \end{equation}
    where each $E_j$ is a Lebesgue measurable subset of $(a,b)$ and $e_j \in \C$, $j=1,\dots,m$.
  \item A function $u:(a,b) \to H$ is \emph{strongly measurable} if there exists a sequence of simple functions $\varsigma_k:(a,b) \to H$ such that
    \begin{equation*}
      \varsigma_k(t) \to u(t) \quad \text{for a.e. } a < t < b.
    \end{equation*}
  \item We define $L^2(a,b;H)$ as the set of strongly measurable functions $u:(a,b) \to H$ such that
    \begin{equation*}
      \| u \|_{L^2(a,b;H)} := \left( \int_a^b \| u(t) \|^2 \,  dt \right)^{1/2} < \infty.
    \end{equation*}
    This is a separable complex Hilbert space.
  \end{enumerate}

Let us remark that every $u \in L^2(a,b;H)$ can be approximate with smooth (in time) functions of the form
\begin{equation} \label{eq:simple_funt_2}
  \varsigma(t) = \sum_{j=1}^m \varphi_j(t) w_j, 
\end{equation}
where $\varphi_j \in C_c^\infty(a,b)$, instead of characteristics (as in (\ref{eq:simple_funct_1})). In what follows, we use $S$ to denote the set of \emph{simple functions}, that is, functions of the form (\ref{eq:simple_funct_1}) or (\ref{eq:simple_funt_2}).  

\begin{Definition} \label{Def:frac_space-time}
  For every $\alpha \ge 0$ and $\varsigma \in S$, with
\begin{equation*}
  \varsigma(t) = \sum_{j=1}^m \varphi_j(t) w_j,
\end{equation*}
we define the norm
\begin{align*}
  \| \varsigma \|_{H_0^\alpha(a,b;H)} &:= \left( \| \varsigma \|_{L^2(a,b;H)}^2 + \| \prescript{}{0}{D}_t^\alpha \varsigma \|_{L^2(a,b;H)}^2 \right)^{1/2} \\
                            & = \left( \sum_{j=1}^m \| \varphi_j \|_\alpha^2 \| w_j \|^2  \right)^{1/2},
\end{align*}
where $\| \varphi_j \|_\alpha$ is the norm considered in (\ref{eq:norm}) and (\ref{eq:norm_2}). Let $\alpha \ge 0$, $(a,b)$ be an open interval of $\R$, where $a \in [-\infty,\infty)$ and $b \in (-\infty,\infty]$ with $a<b$, and $H$ be a separable complex Hilbert space. We define the fractional space $H_0^\alpha(a,b;H)$ as the closure of the set $S$ under the norm $\| \cdot \|_{H_0^\alpha(a,b;H)}$.
\end{Definition}

The fractional space $H_0^\alpha(a,b;H)$ is a Hilbert subspace of $L^2(a,b;H)$ with the inner
product and norm inherit under the limit of elements in $S$. As in the previous section, this space and its inner product and norm are independent of the approximating sequence in $S$ and the orthogonal basis $\{ w_k \}_{k=1}^\infty$. This approach is similar to the one considered in Definition~\ref{Def:fractional_interval}.

As we mentioned in the previous section, we are mainly interested in fractional spaces of order between 0 and 1/2. For this range of fractions we have analogous properties to those in Theorem~\ref{thm:H-alpha} and Lemma~\ref{lm:regularity}; the proof of these properties are similar and we leave the details to the reader. As before, we denote by $\widetilde{u}$ the continuation of $u$ by zero outside of $(a,b)$.

\begin{Theorem}\label{thm:H-alpha-time}
    For every $\alpha \in (0,1)$ the following properties are equivalent:
  \begin{enumerate}[label=(\roman*)]
  \item $u \in H_0^{\alpha/2}(a,b;H)$;
  \item $\widetilde {u} \in H^{\alpha/2}(\R;H)$;
  \item there exists ${U} \in H^{\alpha/2}(\R;H)$ such that restricted to $(a,b)$ is equal to $u$.
  \end{enumerate}
\end{Theorem}

Let us recall the classical time Sobolev space of order one
\begin{equation}\label{eq:classical_time_Sobolev}
  H^1(a,b;H) := \{ u \in L^2(a,b;H): u' \in L^2(a,b;H) \},
\end{equation}
for more details see Section~5.9 of \cite{Ev2010}.
\begin{Lemma} \label{lm:space-time_regularity}
  For every $\alpha \in (0,1)$ we have the contention $H^1(a,b;H) \subset H_0^{\alpha/2}(a,b;H)$.
\end{Lemma}

We observe that this lemma implies that in any \emph{bounded} interval every $u \in C^1([a,b];H)$ belongs to the space $u \in H_0^{\alpha/2}(a,b;H)$. Thereby the functions in this fractional space leave no trace on the extremes of  interval endpoints, so they are not necessarily equal to zero there.

\begin{Definition}
  \begin{enumerate}[label=(\roman*)]
  \item If $\alpha \in (0,1)$ and $u \in H_0^{\alpha/2}(a,b;H)$, we define the Riemann-Liouville fractional derivatives as
\begin{equation} \label{eq:R-L_der}
  \prescript{}{a}{D}_t^{\alpha/2} u := \lim_{k \to \infty} \prescript{}{-\infty}{D}_t^{\alpha/2} \varsigma_k \quad \text{and} \quad \prescript{}{t}{D}_b^{\alpha/2} u := \lim_{k \to \infty} \prescript{}{t}{D}_\infty^{\alpha/2} \varsigma_k,
\end{equation}
where $\{\varsigma_k\}_{k=1}^\infty \subset S$ is a Cauchy sequence with the norm $\| \cdot \|_{H_0^{\alpha/2}(a,b;H)}$, such that $\varsigma_k \to u$ in $L^2(a,b;H)$.
\item If $\alpha > 0$, $(a,b)$ is a \emph{bounded interval} and  $u \in L^2(a,b;H)$ we define the Riemann-Liouville fractional integrals 
\begin{equation} \label{eq:R-L_int}
  \prescript{}{a}{I}_t^\alpha u := \lim_{k \to \infty} \prescript{}{-\infty}{I}_t^\alpha \varsigma_k \quad \text{and} \quad \prescript{}{t}{I}_b^\alpha u := \lim_{k \to \infty} \prescript{}{t}{I}_\infty^\alpha \varsigma_k,
\end{equation}
where $\varsigma_k \to u$ in $L^2(a,b;H)$.
  \end{enumerate}  
\end{Definition}

The limits in (\ref{eq:R-L_der}) are taken in the $L^2(\R;H)$ sense and, thereby, both fractional derivatives belong to this space. Moreover, as in the previous section these fractional operators are well-defined in the sense that they don’t depend nor on the approximating sequence neither on the orthogonal basis $\{ w_k \}_{k=1}^\infty$. On the other hand, by the continuity of the Riemann-Liouville fractional integral operator in the $L^2(a,b)$ time space, see Property~\ref{Prop:App-bound}, the space-time operators in (\ref{eq:R-L_int}) are good defined, i.e. they don't depend nor on the approximating sequence neither on the orthogonal basis $\{ w_k \}_{k=1}^\infty$.

% -------------------------------------------
% -------------------------------------------

\section{Existence, uniqueness and \emph{a priori} estimates} \label{sec:existence}

In this section we implement a Galerkin method in fractional spaces to prove Theorem~\ref{thm:main}. Along this section we prove also important energy estimates that, in the end, imply the \emph{a priori} estimate (\ref{eq:energy_est}).

Lets start with the notion of vector fields. Given a  complex functional space $X$  and a natural number $n$, we denote by ${\bf X}$ the $n$-dimensional space of \emph{vector fields} $X^n$. In particular, if $H$ is a complex Hilbert space with inner product $(\cdot,\cdot)$ and norm $\| \cdot \|$, then ${\bf H}$ is a Hilbert space with inner product and norm
\begin{equation} \label{eq:vector_inner_prod}
  (u,v)_{\bf H} := \sum_{i=1}^n (u_i,v_i), \quad  \| u \|_{\bf H} = \left( \sum_{i=1}^n \| u_i \|^2 \right)^{1/2},
\end{equation}
respectively, where $u=(u_1,\dots,u_n)$ and $v=(v_1,\dots,v_n)$ belong to ${\bf H}$.

Fixed $\alpha \in (0,1)$, we recall the fractional initial boundary problem we are interested:
\begin{equation} \label{eq:main_2}
  \left\{
\begin{aligned}
    &\rho(x) \frac{\partial^2 u}{\partial t^2} - \frac{\partial}{\partial x_i} \left[ B_{ij}(x) \frac{\partial}{\partial x_j} \prescript{C}{0}{D}_t^\alpha u \right] - \frac{\partial}{\partial x_i} \left[ A_{ij}(x) \frac{\partial u}{\partial x_j} \right] = f(x,t), \\
    &\hphantom{aaaaaaaaaaaaaaaaaaaaaaaaaaaaaaaaaaaaaaaa}(x,t) \in \Omega \times (0,T), \\
    &u(x,t) = 0, \quad (x,t) \in \partial \Omega \times (0,T), \\
    &u(x,0) = g(x), \quad \frac{\partial u}{\partial t}(x,0) = h(x), \quad x \in \Omega, 
  \end{aligned}
  \right.
\end{equation}
where the summation over the repeating indices is assumed. Here $\Omega$ is a bounded domain of $\R^n$ with Lipschitz boundary and $u$ the unknown displacement vector field. The operator $\prescript{C}{0}{D}_t^\alpha$ is the Caputo fractional derivative in time
\begin{equation*}
\prescript{C}{0}{D}_t^\alpha u = \prescript{}{0}{D}_t^\alpha (u-g),
\end{equation*}
for smooth $u$ and $g$, see (\ref{eq:Caputo_frac_der}). Let us also assume the following hypothesis:

\begin{enumerate}[label=(H.\arabic*)]
\item \label{H1} The elastic and fractional viscosity tensors are matrix valued funtions: $A_{ij} = (A_{ij}^{kl})_{1 \le k,l \le n} \in \R^{n^2 \times n^2}$, $B_{ij} = (B_{ij}^{kl})_{1 \le k,l \le n} \in \R^{n^2 \times n^2}$, and they are measurable functions of their arguments. We also assume the symmetry on these tensors:
  \begin{equation*}
    A_{ij}^{kl} = A_{ji}^{kl} = A_{ij}^{lk} = A_{kl}^{ij}, \quad , B_{ij}^{kl} = B_{ji}^{kl} = B_{ij}^{lk} = B_{kl}^{ij} \quad 1 \le i,j,k,l \le n.
  \end{equation*}
  Additionally, we suppose that $A$ and $B$ are uniformly elliptic, i.e. there exists $\nu > 0$ such that for all symmetric matrices $(\eta_{ij}) \in \R^{n \times n}$
  \begin{gather*}
    \nu \eta_{ij}^2 \le A_{ij}^{kl}(x) \eta_{ij} \eta_{kl} \le \frac{1}{\nu} \eta_{ij}^2 \quad \text{and} \\
   \nu \eta_{ij}^2 \le B_{ij}^{kl}(x) \eta_{ij} \eta_{kl} \le \frac{1}{\nu} \eta_{ij}^2,
  \end{gather*}
  a.e. $x \in \Omega$.
\item \label{H2} The mass density function $\rho$ is measurable and there exists $\rho_0>0$ such that
  \begin{equation*}
    \rho_0 \le \rho(x) \le \frac{1}{\rho_0}, \quad a.e.\ x \in \Omega. 
  \end{equation*}
\item \label{H3} $f \in L^2(0,T;\bH^{-1}(\Omega))$, $g \in \bH_0^1(\Omega)$ and $h \in \bL^2(\Omega)$.
\end{enumerate}
Lets recall the usual notation $\bH^{-1}(\Omega)$ for the dual of the space $\bH_0^1(\Omega)$, where every action of one of its elements $f \in \bH^{-1}(\Omega)$ on $v \in \bH_0^1(\Omega)$ is represented by $\langle f, v \rangle$.

We'll use a variational method to study (\ref{eq:main_2}) that requires a suitable notion of weak solution. This notion can be motivated by the fractional integration by parts formula (\ref{eq:frac_var_2}). Before stating the definition of weak solutions to (\ref{eq:main_2}), recall the fractional space $H_0^{\alpha/2}(0,T;\bH_0^1(\Omega))$ in Definition~\ref{Def:frac_space-time}. Let us also recall (\ref{eq:vector_inner_prod}) for the inner product of the real Hilbert space of vector fields $\bL^2(\Omega)$:
\begin{equation*}
  (u,v) := \sum_{i=1}^n \int_\Omega u_i v_i \, dx,
\end{equation*}
where $u=(u_1,\dots,u_n)$, $u_i \in L^2(\Omega)$, and  $v=(v_1,\dots,v_n)$, $v_i \in L^2(\Omega)$, $i=1,\dots,n$, are vector fields.

\paragraph{Notion of weak solution for (\ref{eq:main_2}).} A function $u \in H_0^{\alpha/2}(0,T;\bH_0^1(\Omega))$ is a weak solution of (\ref{eq:main_2}) if
\begin{enumerate}[label=(\roman*)]
\item $\frac{\partial u}{\partial t} \in L^2(0,T;\bL^2(\Omega))$;
\item for every vector field $\varphi \in \bC^1(\overline{\Omega} \times [0,T])$, with $\varphi(x,t) = 0$ if $x \in \partial \Omega$ or $t=T$,
    \begin{multline}\label{eq:weak_sol}
      \int_0^T \left( \rho \frac{\partial u}{\partial t} , \frac{\partial \varphi}{\partial t} \right) \, dt - \int_0^T \left[ \left( B_{ij} \prescript{C}{0}{D}_t^{\alpha/2} \frac{\partial u}{\partial x_j} , \prescript{}{t}{D}_T^{\alpha/2} \frac{\partial \varphi}{\partial x_i}  \right) + \left( A_{ij} \frac{\partial u}{\partial x_j} , \frac{\partial \varphi}{\partial x_i} \right) \right] \, dt \\
      = -\int_0^T \langle f,\varphi \rangle \, dt - (h,\varphi(x,0)),
    \end{multline}
    where, as before, the summation over the repeating indices is assumed and we define, motivated by (\ref{eq:Caputo_frac_der}),
    \begin{equation} \label{eq:weak_Caputo_frac_der}
      \prescript{C}{0}{D}_t^{\alpha/2} \frac{\partial u}{\partial x_j} := \prescript{}{0}{D}_t^{\alpha/2} \left( \frac{\partial u}{\partial x_j} - \frac{\partial g}{\partial x_j} \right),
    \end{equation}
    with the fractional Riemann-Liouville derivative $\prescript{}{0}{D}_t^{\alpha/2}$ given by (\ref{eq:R-L_der});
\item $u(0) = g$.
\end{enumerate}

\begin{Remark} \label{rm:weak_sol}
  \begin{enumerate}[label=(\roman*)]
  \item Under the hypothesis above on $u$ and $g$, it can be shown that $u \in C([0,T];\bL^2(\Omega))$ (see Section~5.9.2 in \cite{Ev2010}), thus the equality (iii) makes sense.
  \item In a classical framework: $f \in \bC(\overline{\Omega} \times [0,T])$, $g,h \in \bC^1(\overline{\Omega})$, every classical solution $u \in \bC^2(\overline{\Omega} \times [0,T])$ of (\ref{eq:main_2}) is also a weak solution. Indeed, by Lemma~\ref{lm:space-time_regularity} we have $u \in H_0^{\alpha/2}(0,T;\bH_0^1(\Omega))$. Thereby only left to verify (\ref{eq:weak_sol}) for every function $\varphi \in \bC^1(\overline{\Omega} \times [0,T])$ with $\varphi(x,t) = 0$ if $x \in \partial \Omega$ or $t=T$.  This can be done by multiplying the fist equation of (\ref{eq:main_2}) by $\varphi$ and then integrate by parts in space, in the usual way, and in time with formula (\ref{eq:frac_var_2}); recall also Lemma~\ref{lm:frac_der}.
  \item By an approximaption argument, it is not difficult to show that (\ref{eq:weak_sol}) is still true for $\varphi \in H^1(0,T;\bH_0^1(\Omega))$ such that $\varphi(x,T)=0$ for all $x \in \Omega$; recall (\ref{eq:classical_time_Sobolev}) for the definition of this space. 
  \end{enumerate}
\end{Remark}

% -------------------------------------------

\subsection{Galerkin approximations}
\label{sec:galerk-appr}

To find weak solutions to (\ref{eq:main_2}) we use the Galerkin method: to construct 
solutions of certain finite-dimensional approximations of the problem and then pass to the limit (see Sections 7.1 and 7.2 in \cite{Ev2010}.)

Let $\{ w_k \}_{k=1}^\infty$ be an \emph{orthogonal basis} of $\bH_0^1(\Omega)$ that, additionally, is an \emph{orthonormal basis} of $\bL^2(\Omega)$. One of such an orthogonal basis can be assembled with the eigenfuntions of the Laplace operator in the bounded domain $\Omega$.

Fix a positive integer $m$ and consider the function
\begin{equation}\label{eq:u_m}
  u_m(t) := \sum_{k=1}^m d_m^k(t)w_k,
\end{equation}
We select next the coefficients $d_m^k(t)$ to make $u_m$ an approximate solution to (\ref{eq:main_2}) in the following sense: $d_m^k(t)$, $k=1,\dots,m$, satisfy the following projection of (\ref{eq:main_2}) onto the finite-dimensional subspace spanned by $\{w_k \}_{k=1}^m$: for $k=1,\dots,m$,
\begin{align}
  (\rho u_m'',w_k) + \left( B_{ij} \prescript{C}{0}{D}_t^\alpha \frac{\partial u_m}{\partial x_j}, \frac{\partial w_k}{\partial x_i} \right) + \left( A_{ij} \frac{\partial u_m}{\partial x_j}, \frac{\partial w_k}{\partial x_i} \right) &= \langle f,w_k \rangle, & & t \in (0,T), \label{eq:approx} \\
  d_m^k(0) &= (g,w_k), \label{eq:approx_bound} \\
  {d_m^k}'(0) &= (h,w_k), \label{eq:approx_initial}
\end{align}
where the summation over the indices $i,j=1,\dots,n$ is assumed. This system of ordinary fractional differential equations can be solved by transforming it in a Volterra type equation to find a unique solution $u_m \in C^2([0,T];\bH_0^1(\Omega))$ of the form (\ref{eq:u_m}) that satisfies (\ref{eq:approx})-(\ref{eq:approx_initial}), see Property~\ref{Prop:App-Volterra} in the Appendix.

% -------------------------------------------

\subsection{Energy estimates}
\label{sec:energy-estimates}

Before sending $m \to \infty$ in the Galerking approximations, we need to control some important norms of the approximate solutions $u_m$. In fact, we will prove that for every $m \in \N$, there exists $C=C(\nu,\rho_0,\Omega,T,\alpha)$ such that
\begin{multline} \label{eq:approx_energy_est}
  \max_{0 \le t \le T} \left( \| u_m(t) \|_{\bH_0^1(\Omega)} + \| u_m'(t) \|_{\bL^2(\Omega)} \right) + \| u_m \|_{H_0^{\alpha/2}(0,T;\bH_0^1(\Omega))} \\
  \le C \left( \| f \|_{L^2(0,T;\bH^{-1}(\Omega))} + \| g \|_{\bH_0^1(\Omega)} + \| h \|_{\bL^2(\Omega)} \right),
\end{multline}
In what follows, $C$ represents a positive constant that \emph{could change} from one inequality to the next, but only depends on $\nu,\rho_0,\Omega,T,\alpha$.

To estimate the first term of the left-hand side of (\ref{eq:approx_energy_est}), we multiply (\ref{eq:approx}) by ${d_m^k}'$ and sum $k=1,\dots,m$ to find
\begin{equation*}
  (\rho u_m'',u_m') + \left( B_{ij} \prescript{C}{0}{D}_t^\alpha \frac{\partial u_m}{\partial x_j}, \frac{\partial u_m'}{\partial x_i} \right) + \left( A_{ij} \frac{\partial u_m}{\partial x_j}, \frac{\partial u_m'}{\partial x_i} \right) = \langle f,u_m' \rangle, \quad t \in (0,T).
\end{equation*}
Observe that this equation can be written in the following way (recall (\ref{eq:frac_der_Caputo}))
\begin{equation} \label{eq:energy_variation}
  E'(t) + \left( B_{ij} \prescript{}{0}{I}_t^{1-\alpha} \frac{\partial u_m'}{\partial x_j}, \frac{\partial u_m'}{\partial x_i} \right) = \langle f,u_m' \rangle,
\end{equation}
where $E(t)$ is the total energy of $u$ at time $t$, i.e. the sum of the kinetic and elastic energies,
\begin{equation*}
  E(t) := \frac{1}{2} \| \sqrt{\rho} u_m'(t) \|_{\bL^2(\Omega)}^2 + \frac{1}{2} \left( A_{ij} \frac{\partial u_m}{\partial x_j}(t), \frac{\partial u_m}{\partial x_i}(t) \right).
\end{equation*}
For a given $0 \le t \le T$, we integrate (\ref{eq:energy_variation}) in the time interval $(0,t)$ to obtain
\begin{equation} \label{eq:energy_balance}
  E(t)-E(0) = \int_0^t \langle f,u_m' \rangle \, ds - \int_0^t \left( B_{ij} \prescript{}{0}{I}_s^{1-\alpha} \frac{\partial u_m'}{\partial x_j}, \frac{\partial u_m'}{\partial x_i} \right) \, ds,
\end{equation}
This expression can be seen as an energy balance: if we consider the total energy of $u$ as the sum of its kinetic and elastic energies, then the change of this quantity, along the time interval $(0,t)$, is measured by the work done by the external force $f$ minus the energy dissipated by the fractional viscosity, represented by the last term in the right-hand side of the previous equation; next lemma shows that in fact this quantity is nonnegative.

\begin{Lemma} \label{lm:viscosity_energy}
  For every $t \in [0,T]$ and $u \in L^2(0,T;\bH_0^1(\Omega))$ we have
  \begin{equation} \label{eq:viscosity_energy}
    \int_0^t \left( B_{ij} \prescript{}{0}{I}_s^{1-\alpha} \frac{\partial u}{\partial x_j}, \frac{\partial u}{\partial x_i} \right) \, ds \ge 0,
  \end{equation}
  where $\prescript{}{0}{I}_s^{1-\alpha}$ is the Riemann-Liouville fractional integral defined in (\ref{eq:R-L_int}).
\end{Lemma}

\begin{proof}
  Observe first that, by hypothesis (H1) and the classical Poincaré inequality,
  \begin{equation*}
    (u,v)_B := \left( B_{ij} \frac{\partial u}{\partial x_j}, \frac{\partial v}{\partial x_i} \right)
  \end{equation*}
  generates an equivalent inner product to the usual one of $H_0^1(\Omega)$. So let us select a orthogonal basis $\{ w_k^B \}_{k=1}^\infty$ of $H_0^1(\Omega)$ with this new inner product. By (\ref{eq:R-L_int}) and the comments after, it is enough to prove (\ref{eq:viscosity_energy}) for $u$ of the form
  \begin{equation*}
    u(x,t) = \sum_{k=1}^m \varphi_k(t) w_k^B,
  \end{equation*}
  where $\varphi_j \in C_c^\infty(a,b)$. In this setting, (\ref{eq:viscosity_energy}) can be written as
  \begin{equation*}
    \sum_{k=1}^m \left( \int_0^t \prescript{}{0}{I}_s^{1-\alpha} \varphi_k(s) \varphi_k(s) \, ds  \right) \left( w_k^B,w_k^B \right)_B \ge 0,
  \end{equation*}
  which is true by (\ref{eq:frac_var_I}) and (\ref{eq:energy_equiv_I}).
\end{proof}

The previous lemma together with (\ref{eq:energy_balance}), hypothesis (H1) and (H2), Cauchy-Schwarz inequality and boundary conditions (\ref{eq:approx_bound}), (\ref{eq:approx_initial}) imply that there exists a positive constant $C=C(\nu,\rho_0)$ such that for every  $0 \le t \le T$ we have
\begin{multline}\label{eq:max_est}
\max \left\{ \| u_m'(t) \|_{\bL^2(\Omega)}^2, \left\| \frac{\partial u}{\partial x_i}(t) \right\|_{\bL^2(\Omega)}^2 \right\} \\
  \le C \left( \| f \|_{L^2(0,T;\bH^{-1}(\Omega))}^2 + \| g \|_{\bH_0^1(\Omega)}^2 + \| h \|_{\bL^2(\Omega)}^2 + \int_0^t \| u_m'(s) \|_{\bL^2(\Omega)}^2 \, ds \right).
\end{multline}
This estimate and the Gr\"{o}nwall's inequality imply the existence of a positive constant $C=C(\nu,\rho_0,T)$ such that
\begin{equation} \label{eq:max_est_u_prima}
  \max_{0 \le t \le T} \| u_m'(t) \|_{\bL^2(\Omega)} \le C \left( \| f \|_{L^2(0,T;\bH^{-1}(\Omega))} + \| g \|_{\bH_0^1(\Omega)} + \| h \|_{\bL^2(\Omega)} \right).
\end{equation}
On the other hand, by (\ref{eq:max_est}) and Poincar\'e inequality, there exists $C=C(\nu,\rho_0,\Omega,T)$ such that
\begin{equation} \label{eq:max_est_u}
    \max_{0 \le t \le T} \| u_m(t) \|_{\bH_0^1(\Omega)} \le C \left( \| f \|_{L^2(0,T;\bH^{-1}(\Omega))} + \| g \|_{\bH_0^1(\Omega)} + \| h \|_{\bL^2(\Omega)} \right).
  \end{equation}

  Let us now estimate the second term in the left-hand side of (\ref{eq:approx_energy_est}). To this end, multiply (\ref{eq:approx}) by $d_m^k - (g,w_k)$ and sum $k=1,\dots,m$, to obtain
\begin{multline*}
  (\rho u_m'',u_m - g_m) + \left( B_{ij} \prescript{C}{0}{D}_t^\alpha \frac{\partial u_m}{\partial x_j}, \frac{\partial}{\partial x_i}( u_m - g_m) \right) + \left( A_{ij} \frac{\partial u_m}{\partial x_j}, \frac{\partial u_m}{\partial x_i} - \frac{\partial g_m}{\partial x_i} \right) \\
  = \langle f,u_m-g_m \rangle, \quad t \in (0,T),
\end{multline*}
where $g_m := \sum_{k=1}^m (g,w_k)w_k$, the projection of $g$ on the finite-dimensional subspace spanned by $w_1,\dots,w_m$. We now integrate previous identity in the time-interval $(0,T)$ and then use integration by parts to find that
\begin{multline*}
  \int_0^T \left( B_{ij} \prescript{}{0}{D}_t^\alpha \frac{\partial}{\partial x_j} (u_m-g_m), \frac{\partial}{\partial x_i}( u_m - g_m) \right) dt + \int_0^T \left( A_{ij} \frac{\partial u_m}{\partial x_j}, \frac{\partial u_m}{\partial x_i}\right) dt \\
  = -(\rho u_m'(T),u_m(T)-g_m) + \int_0^T \left[ (\rho u_m',u_m') + \left( A_{ij} \frac{\partial u_m}{\partial x_j}, \frac{\partial g_m}{\partial x_i} \right)  + \langle f,u_m-g_m \rangle \right] dt.
\end{multline*}
Therefore, by the previous equation, hypothesis (H1) and (H2), Cauchy-Schwarz inequality and (\ref{eq:max_est_u_prima}), (\ref{eq:max_est_u}), there exists a positive constant $C=C(\nu,\rho_0,\Omega,T)$ such that
\begin{multline*}
  \int_0^T \left( \prescript{}{0}{D}_t^\alpha \frac{\partial}{\partial x_j} (u_m-g_m), \frac{\partial}{\partial x_i}( u_m - g_m) \right) dt \\
  \le C \left( \| f \|_{L^2(0,T;\bH^{-1}(\Omega))}^2 + \| g \|_{\bH_0^1(\Omega)}^2 + \| h \|_{\bL^2(\Omega)}^2 \right).
\end{multline*}
To conclude, observe that previous equation together with (\ref{eq:frac_var_2}), (\ref{eq:energy_equiv_D}) and the fractional Poincaré inequality, Theorem~\ref{thm:frac_Poincare}, imply that there exists a positive constant $C=C(\nu,\rho_0,\Omega,T)$ such that
\begin{equation*}
  \|u_m-g_m \|_{H_0^{\alpha/2}(0,T;\bH_0^1(\Omega))} \le \frac{C}{\cos(\alpha \pi/2)} \left( \| f \|_{L^2(0,T;\bH^{-1}(\Omega))} + \| g \|_{\bH_0^1(\Omega)} + \| h \|_{\bL^2(\Omega)} \right).
\end{equation*}
This not only finishes the proof of (\ref{eq:approx_energy_est}), but gives an interesting dependence of the estimate on the fractional order of differentiation $\alpha$. In fact, observe that previous estimate grows uncontrollably as $\alpha \to 1^-$.

% -------------------------------------------

\subsection{Existence of a weak solution}
\label{sec:exist-weak-solut}

With the energy estimate (\ref{eq:approx_energy_est}) we can take limit as $m \to \infty$ in (\ref{eq:approx})-(\ref{eq:approx_initial}). Indeed, observe that (\ref{eq:approx_energy_est}) implies the sequences $\{ u_m \}_m$, $\{u_m'\}_m$ and $\{u_m-g_m\}_m$ are bounded in $L^2(0,T;\bH_0^1(\Omega))$, $L^2(0,T;\bL^2(\Omega))$ and $H_0^{\alpha/2}(0,T;\bH_0^1(\Omega))$, respectively. Therefore, there exists a subsequence $\{ u_{m_l} \}_l$ and $u \in H_0^{\alpha/2}(0,T;\bH_0^1(\Omega))$, with $u' \in L^2(0,T;\bL^2(\Omega))$, such that
\begin{equation} \label{eq:weak_conv}
  \left\{
  \begin{aligned}
    & u_{m_l} \rightharpoonup u \in L^2(0,T;\bH_0^1(\Omega)), \\
    & \prescript{C}{0}{D}_t^{\alpha/2} u_{m_l} \rightharpoonup \prescript{C}{0}{D}_t^{\alpha/2} u \in L^2(0,T;\bH_0^1(\Omega)) \quad \text{and}\\
    & u_{m_l}' \rightharpoonup u' \in L^2(0,T;\bL^2(\Omega)).
  \end{aligned}
  \right.
\end{equation}
On the other hand, lets fix $N \in \N$ and consider $\varphi \in C^1([0,T];\bH_0^1(\Omega))$ as follows
\begin{equation}\label{eq:phi_sep}
  \varphi(x,t) = \sum_{k=1}^N d^ k(t) w_k,
\end{equation}
where $\{ d^k \}_{k=1}^N \subset C^1[0,T]$ and $d_k(T)=0$ for $k=1,\dots,N$. For $m \ge N$, multiply (\ref{eq:approx}) by $d^k$, sum $k=1,\dots,m$ and then integrate by parts in space and time, as in Remark~\ref{rm:weak_sol}$(ii)$, to obtain
\begin{equation} \label{eq:weak_sol_approx}
    \begin{aligned}
      \int_0^T \left( \rho \frac{\partial u_{m_l}}{\partial t} , \frac{\partial \varphi}{\partial t} \right) \, dt &- \int_0^T \left[ \left( B_{ij} \prescript{C}{0}{D}_t^{\alpha/2} \frac{\partial u_{m_l}}{\partial x_j} , \prescript{}{t}{D}_T^{\alpha/2} \frac{\partial \varphi}{\partial x_i}  \right) + \left( A_{ij} \frac{\partial u_{m_l}}{\partial x_j} , \frac{\partial \varphi}{\partial x_i} \right) \right] \, d t \\
      =& -\int_0^T \langle f,\varphi \rangle \, dt - (h,\varphi(x,0)).
    \end{aligned}
  \end{equation}
  By taking limit as $l \to \infty$ and using (\ref{eq:weak_conv}) we deduce that $u$ satisfies (\ref{eq:weak_sol}) for every $\varphi \in C^1([0,T];\bH_0^1(\Omega))$ of the form (\ref{eq:phi_sep}). To conclude (\ref{eq:weak_sol}) for $\varphi \in \bC^1(\overline{\Omega} \times [0,T])$, with $\varphi(x,t) = 0$ if $x \in \partial \Omega$ or $t=T$, we observe that all of such functions can be approximated by functions of the form (\ref{eq:phi_sep}) with the norm $\| \cdot \|_{H_0^{\alpha/2}(0,T;\bH_0^1(\Omega))}$.

  Only left to prove the initial condition $u(0)=g$. To this end, let us consider (\ref{eq:weak_sol_approx}) for any $\varphi \in C^2([0,T];\bH_0^1(\Omega))$, with $\varphi(x,T)=\varphi'(x,T)=0$ for all $x \in \Omega$. If we again integrate by parts in time and then take $l \to \infty$, we deduce that
  \begin{multline} \label{eq:weak_sol_2}
    \int_0^T \left( \rho u , \frac{\partial^2 \varphi}{\partial t^2} \right) \, dt - \int_0^T \left[ \left( B_{ij} \prescript{C}{0}{D}_t^{\alpha/2} \frac{\partial u}{\partial x_j} , \prescript{}{t}{D}_T^{\alpha/2} \frac{\partial \varphi}{\partial x_i}  \right) + \left( A_{ij} \frac{\partial u}{\partial x_j} , \frac{\partial \varphi}{\partial x_i} \right) \right] \, dt \\
    = -\int_0^T \langle f,\varphi \rangle \, dt - (h,\varphi(x,0)) + (g,\varphi'(x,0)).
  \end{multline}
  On the other hand, if we integrate by parts (\ref{eq:weak_sol}) we have
  \begin{multline*}
    \int_0^T \left( \rho u , \frac{\partial^2 \varphi}{\partial t^2} \right) \, dt - \int_0^T \left[ \left( B_{ij} \prescript{C}{0}{D}_t^{\alpha/2} \frac{\partial u}{\partial x_j} , \prescript{}{t}{D}_T^{\alpha/2} \frac{\partial \varphi}{\partial x_i}  \right) + \left( A_{ij} \frac{\partial u}{\partial x_j} , \frac{\partial \varphi}{\partial x_i} \right) \right] \, dt \\
    = -\int_0^T \langle f,\varphi \rangle \, dt - (h,\varphi(x,0)) + (u(0),\varphi'(x,0)).
  \end{multline*}
  By comparing this expression with (\ref{eq:weak_sol_2}), we conclude that $u(0)=g$, as $\varphi$ is arbitrary.

  Notice that the above solutions satisfies energy estimate (\ref{eq:energy_est}), as a consequence of taking $m \to \infty$ in (\ref{eq:approx_energy_est}) together with (\ref{eq:weak_conv}).

% -------------------------------------------

\subsection{Uniqueness}
\label{sec:uniqueness}

In order to prove that the above solution is unique, it is sufficient to verify that $u=0$ is the unique weak solution of (\ref{eq:main_2}) with $f \equiv g \equiv h \equiv 0$. Suppose then that $u$ is such a weak solution and fix $s \in (0,T)$. Le us consider
  \begin{equation*}
    v(x,t) =
    \begin{cases}
      \int_t^s u(x,\tau) \, d\tau & \text{if } 0 \le t \le s, \\
      0 & \text{if } s \le t \le T.
    \end{cases}
  \end{equation*}
  We can substitute $v$ in (\ref{eq:weak_sol}), by Remark~\ref{rm:weak_sol} $(iii)$, to find that  
  \begin{equation*}
    \int_0^s \left( \rho \frac{\partial u}{\partial t} , \frac{\partial v}{\partial t} \right) \, dt - \int_0^s \left( B_{ij} \prescript{C}{0}{D}_t^{\alpha/2} \frac{\partial u}{\partial x_j} , \prescript{}{t}{D}_s^{\alpha/2} \frac{\partial v}{\partial x_i}  \right)\, dt - \int_0^s \left( A_{ij} \frac{\partial u}{\partial x_j} , \frac{\partial v}{\partial x_i} \right) \, dt = 0.
  \end{equation*}

  Observe that $v_t = -u$ and $v(s)=0$, so we can transform the previous equation in
  \begin{equation} \label{eq:uniqueness}
    \frac{1}{2} \| \sqrt{\rho} u(s) \|_{\bL^2(\Omega)}^2 + \int_0^s \left( B_{ij} \prescript{}{0}{D}_t^{\alpha/2} \frac{\partial u}{\partial x_j} , \prescript{}{t}{D}_s^{\alpha/2} \frac{\partial v}{\partial x_i}  \right) \, dt + \frac{1}{2} \left( A_{ij} \frac{\partial v}{\partial x_j}(0) , \frac{\partial v}{\partial x_i}(0) \right) = 0.
  \end{equation}
  The third term of the left-hand side of this equation is nonnegative, by Hypothesis (H1); let us see that the second one is also nonnegative. To this end, notice that we can suppose that $u$ is of the form (\ref{eq:simple_funt_2}), as $u \in H_0^{\alpha/2}(0,T;\bH_0^1(\Omega))$ is a limit of these functions, by Definition~\ref{Def:frac_space-time}. Thereby, suppose $u$ have such a form, so we can compute freely and use (\ref{eq:frac_var_2}) to obtain
  \begin{align*}
    \int_0^s \left( B_{ij} \prescript{}{0}{D}_t^{\alpha/2} \frac{\partial u}{\partial x_j} , \prescript{}{t}{D}_s^{\alpha/2} \frac{\partial v}{\partial x_i}  \right) \, dt &= \int_0^s \left( B_{ij}\frac{\partial u}{\partial x_j} , \prescript{}{t}{D}_s^\alpha \frac{\partial v}{\partial x_i}  \right) \, dt \\
    &= \int_0^s \left( B_{ij}\frac{\partial u}{\partial x_j} , \prescript{}{t}{I}_s^{1-\alpha} \frac{\partial u}{\partial x_i}  \right) \, dt,
  \end{align*}
  and this last integral is nonnegative by Lemma~\ref{lm:viscosity_energy}. We conclude then from (\ref{eq:uniqueness}) that $u(s)=0$, and as $s \in (0,T)$ is arbitrary $u \equiv 0$.

% -------------------------------------------
% -------------------------------------------
  
% \section*{Acknowledgements}

% -------------------------------------------
% -------------------------------------------

\appendix

\section{Appendix}

In this appendix we collect several useful properties of fractional operators used throughout the article; we provide some of their proofs for completeness. For more details see \cite{SaKiMa93}. In what follows $\alpha \in (0,1)$ and $(a,b)$ is a \emph{bounded interval}.  

Let us start by recalling the Fourier transform and its inverse:
\begin{gather*}
  \F(u)(\omega) = \hat u (\omega) := \int_{-\infty}^\infty u(t) e^{-i \omega t} \, dt, \\
  \F^{-1}(u)(t) = \check u (t) := \int_{-\infty}^\infty u(\omega) e^{-i \omega t} \, d\omega,
\end{gather*}
where $i = \sqrt{-1}$, and their following well known properties:
\begin{equation} \label{eq:App_mult_formula}
    \int_{-\infty}^\infty \hat u v = \int_{-\infty}^\infty u \hat v, \quad \hat u(\omega) = \check u (-\omega).
  \end{equation}
  
The proof of the following important theorem can be found in Section~7.1 of \cite{SaKiMa93}.
  \begin{Property}
    \label{Prop:App-Fourier_deri}
    Given a functions $\varphi \in C_c^\infty(\R)$, the Fourier transform of its Riemann-Liouville fractional derivatives and integrals satisfy:
    \begin{enumerate}[label=(\roman*)]
    \item  if $\alpha \ge 0$,
      \begin{equation*}
        \F (\prescript{}{-\infty}{D}_t^\alpha \varphi)(\omega) = (i \omega)^\alpha \hat \varphi(\omega), \quad \F (\prescript{}{t}{D}_\infty^\alpha \varphi)(\omega) = (-i \omega)^\alpha \hat \varphi(\omega);
      \end{equation*}
\item if $\alpha \in [0,1)$,
  \begin{equation*}
    \F (\prescript{}{-\infty}{I}_t^{\alpha/2} \varphi)(\omega) = \frac{\hat \varphi(\omega)}{(i \omega)^{\alpha/2}}, \quad \F (\prescript{}{t}{D}_\infty^{\alpha/2} \varphi)(\omega) = \frac{\hat \varphi(\omega)}{(-i \omega)^{\alpha/2}}.
  \end{equation*}
\end{enumerate}
\end{Property}

It is important to mention that in $(ii)$ we restrict ourselves to $\alpha \in [0,1)$ to keep all the expressions in the $L^2(\R)$ space. We recall that
\begin{equation} \label{eq:complex_power}
  (i \omega)^\alpha = |\omega|^\alpha e^{\frac{\alpha \pi i}{2} \sgn \omega},
\end{equation}
where $\sgn$ is the sign function.

The proof of the following result can be found in \cite{SaKiMa93}, Theorem 2.6. 

\begin{Property} \label{Prop:App-bound}
  For all $\beta \ge 0$, $\prescript{}{a}{I}_t^\beta$ and $\prescript{}{t}{I}_b^\beta$ are bounded linear operators in $L^2(a,b)$.
\end{Property}

\begin{Property}[Inverse] \label{Prop:App-inverse}
  For all $u \in H_0^{\alpha/2}(a,b)$ we have
  \begin{gather}
    \prescript{}{a}{I}_t^{\alpha/2} \prescript{}{a}{D}_t^{\alpha/2} u = \prescript{}{a}{D}_t^{\alpha/2} \prescript{}{a}{I}_t^{\alpha/2} u = u, \label{eq:App-inverse} \\
    \prescript{}{t}{I}_b^{\alpha/2} \prescript{}{t}{D}_b^{\alpha/2} u = \prescript{}{t}{D}_b^{\alpha/2} \prescript{}{t}{I}_b^{\alpha/2} u = u.    
  \end{gather}
\end{Property}

\begin{proof}
  Observe that it is enough to prove the result for $u \in C_c^\infty(a,b)$, as by definition all the other elements in $H_0^{\alpha/2}(a,b)$ are limit of sequences of these elements (recall also Lemma~\ref{lm:frac_der} and Property~\ref{Prop:App-bound}). So let $u \in C_c^\infty(a,b)$. To prove (\ref{eq:App-inverse}) we just have to take Fourier transform in both sides of the equations and use Property~\ref{Prop:App-Fourier_deri}.
\end{proof}

\begin{Property}[Fractional integration by parts] \label{Prop:App-parts}
  For every $\varphi, \psi \in C_c^\infty(\R)$, we have
  \begin{align}
    \int_{-\infty}^\infty \prescript{}{-\infty}{D}_t^{\alpha/2} \varphi \psi = \int_{-\infty}^\infty \varphi \prescript{}{t}{D}_\infty^{\alpha/2} \psi, \label{eq:App-frac_int_parts} \\
    \int_{-\infty}^\infty \prescript{}{-\infty}{I}_t^{\alpha/2} \varphi \psi = \int_{-\infty}^\infty \varphi \prescript{}{t}{I}_\infty^{\alpha/2} \psi. \label{eq:App-frac_int_parts_2}
  \end{align}
\end{Property}

\begin{proof}
  Let us prove (\ref{eq:App-frac_int_parts}), the proof of (\ref{eq:App-frac_int_parts_2}) is similar. Indeed, by Property~\ref{Prop:App-Fourier_deri} and (\ref{eq:App_mult_formula}) we have
  \begin{align*}
    \int_{-\infty}^\infty \prescript{}{-\infty}{D}_t^{\alpha/2} \varphi \psi &= \int_{-\infty}^\infty \varphi \left( (it)^{\alpha/2} \check \psi  \right)^\wedge \\
    &= \int_{-\infty}^\infty \varphi \left( (-it)^{\alpha/2} \check \psi(-t)  \right)^\vee = \int_{-\infty}^\infty \varphi \prescript{}{t}{D}_\infty^{\alpha/2} \psi.
  \end{align*}
\end{proof}
We observe that (\ref{eq:App-frac_int_parts}) is still true for every $\alpha > 0$.

\begin{Property}[Semigroup property]
  \begin{enumerate}[label=(\roman*)]
  \item Given $\beta, \gamma \ge 0$, if $u \in L^1(a,b)$ then
    \begin{equation} \label{eq:App-semi}
      \prescript{}{a}{I}_t^\beta \prescript{}{a}{I}_t^\gamma u = \prescript{}{a}{I}_t^{\beta + \gamma} u, \quad \prescript{}{t}{I}_b^\beta \prescript{}{t}{I}_b^\gamma u = \prescript{}{t}{I}_b^{\beta+\gamma} u.
    \end{equation}
  \item If $u \in C^1[a,b]$ and $u(a)=0$ then
    \begin{equation} \label{eq:App-simi-der}
      \prescript{}{a}{D}_t^{\alpha/2} \prescript{}{a}{D}_t^{\alpha/2} u = \prescript{}{a}{D}_t^\alpha u;
    \end{equation}
  \end{enumerate}
\end{Property}

\begin{proof}
  The proof of $(i)$ can be found in Section~2.3 of \cite{SaKiMa93}. To prove $(ii)$ we use the fact that $u(a)=0$ together with (\ref{eq:frac_der_AC}) and $(i)$ to find that
  \begin{align*}
    \prescript{}{a}{D}_t^{\alpha/2} \prescript{}{a}{D}_t^{\alpha/2} u &= \frac{d}{dt} \prescript{}{a}{I}_t^{1- \alpha/2} \prescript{}{a}{I}_t^{1- \alpha/2} u' \\
                                                                      & = \frac{d}{dt} \prescript{}{a}{I}_t^{2- \alpha} u' \\
                                                                      &= \prescript{}{a}{I}_t^{1- \alpha} u' = \prescript{}{a}{D}_t^\alpha u.
  \end{align*}
\end{proof}

\begin{Property} \label{Prop:App-Volterra}
  Given $T>0$, vectors $a,b,c,d \in \R^m$ and a vector field $f \in L^1(0,T)^m$, there exists a unique vector field $u \in C^2[0,T]^m$, solution of the fractional ODE
  \begin{equation} \label{eq:App-ODE}
    \left\{
      \begin{gathered}
        u'' + b \prescript{C}{0}{D}_t^\alpha u + au = f(t), \quad t \in (0,T), \\
        u(0) = c, \quad u'(0)= d.
      \end{gathered}
    \right.
  \end{equation}
\end{Property}

\begin{proof}
  To solve (\ref{eq:App-ODE}) we use a classical approach, to transform the equation in a Volterra integral equation of the second kind: indeed, after integrating once (\ref{eq:App-ODE}) and use (\ref{eq:Caputo_frac_der}) we obtain
  \begin{equation}\label{eq:App-der}
    u' = d + \prescript{}{0}{I}_t f - b \prescript{}{0}{I}_t^{1-\alpha}(u-c) - a \prescript{}{0}{I}_t u. 
  \end{equation}
  After integrating once more and using (\ref{eq:App-semi}), we deduce that
  \begin{equation} \label{eq:App-Volterra}
    u(t) = c + dt + \frac{bc}{\Gamma(3-\alpha)} t^{2-\alpha} + \prescript{}{0}{I}_t^2 f(t) - b \prescript{}{0}{I}_t^{2-\alpha} u(t) - a \prescript{}{0}{I}_t^2 u(t), 
  \end{equation}
  where we have used the identity (see Section~2.5 of \cite{SaKiMa93})
  \begin{equation*}
    \prescript{}{0}{I}_t^\beta 1 = \frac{t^\beta}{\Gamma(\beta+1)}, \quad \beta \ge 0.
  \end{equation*}
  The Volterra integral equation (\ref{eq:App-Volterra}) can be solved by a classical fixed point argument, we omit the details. To conclude, observe that this solution belongs to the space $C^2[0,T]^m$, by (\ref{eq:App-der}) and (\ref{eq:App-Volterra}).
\end{proof}

% -------------------------------------------
% -------------------------------------------

%\bibliography{bibliography}
%\bibliographystyle{plain}

\end{document}